\documentclass{amsart}

\newtheorem{theorem}{Theorem}[section]
\newtheorem{lemma}[theorem]{Lemma}
\newtheorem{proposition}[theorem]{Proposition}
\newtheorem{corollary}[theorem]{Corollary}

\theoremstyle{definition}

\theoremstyle{remark}

\newcommand{\what}{\widehat}
\newcommand{\lgra}{\longrightarrow}
\newcommand{\wtilde}{\widetilde}
\newcommand{\R}{\mathbb R}%
\newcommand{\C}{\mathbb C}%
\newcommand{\PP}{\mathcal P}%

\numberwithin{equation}{section}

\begin{document}
\title[Growth of Fourier transform on Damek-Ricci spaces]{A note on growth of Fourier transforms and Moduli of continuity on Damek Ricci spaces}

%

\author[S. K. Ray]{Swagato K. Ray }
\address{Department of Mathematics and Statistics, Indian Institute of Technology, Kanpur 208016, India}
\email{skray@iitk.ac.in}
\thanks{This work is supported by a research
grant no. 48/1/2006-R\&DII/1488 of National Board for Higher Mathematics, India.}

\author[R. P. Sarkar]{Rudra P. Sarkar}
\address{Stat-Math Unit, Indian Statistical
Institute, 203 B. T. Rd., Calcutta 700108, India}
\email{rudra@isical.ac.in}



\begin{abstract} We obtain results related to  boundedness of the growth of Fourier transform by the modulus of continuity on Damek-Ricci spaces. For noncompact riemannian symmetric spaces of rank one, analogues of all  the results follow the same way. 
\end{abstract}

\maketitle

\section{Introduction}
This article is motivated by a recent paper of Bray and Pinsky (\cite{BP}) on growth properties of Fourier transform on euclidean spaces and their analogues on noncompact riemannian symmetric spaces of rank one.
The two main results for symmetric spaces $X$ proved in \cite{BP} are the following (for notation see \cite{BP} and Section 2):
\begin{theorem}
Let $p\in [1,2]$ and $f\in L^p(X)$. Then for $|\eta|<\gamma_p$, $\lambda\in \R$ and $r\ge r_0>0$

$$\sup_\lambda\left[\min\left\{1, \left(\frac{\lambda}r\right)^2\right\}\int_K|\wtilde{f}(\lambda+i\eta,k)|dk\right]\le C_{p, \eta}\Omega_p[f]\left(\frac 1r\right).$$
\end{theorem}

\begin{theorem} For $p\in (1,2)$ let $f$ be a $K$-finite function in $L^p(X)$. Then for $|\eta|<\gamma_p\rho$, $\lambda\in \R$ and $r\ge r_0>0$
$$\left(\int_\R\min\left\{1, \left(\frac{\lambda}r\right)^{2p'}\right\}\int_K|\wtilde{f}(\lambda+i\eta,n)|^{p'}dk|c(\lambda)|^{-2}d\lambda
\right)^{1/p'}\le C_{p,\eta}\Omega_{p}[f]\left(\frac 1r\right).$$
\end{theorem}

The natural question  whether this theorem can be generalized without putting the restriction of $K$-finiteness is asked by the authors in \cite{BP}. They have also conjectured the existence of  analogues of these results for radial functions on harmonic $NA$ groups, which are also known as Damek-Ricci spaces. We shall use both of these names.

We recall that a riemannian symmetric space $X$ of rank one is a quotient space $G/K$ where $G$ is a connected noncompact semisimple Lie group of real rank one with finite centre and $K$ is a maximal compact subgroup of $G$.
We also recall that a harmonic $NA$ group $S$ is a solvable Lie group. The distinguished prototypes of them are the noncompact riemannian symmetric spaces of rank one, which account for a very small subclass of the class of all $NA$ groups (see \cite{ADY}). All the results proved in this article for $NA$ groups will have natural interpretations on symmetric spaces and proving them  will be simpler. Some of the intrinsic difficulties of working with $NA$ groups is the lack of $G$-action and in particular $K$-action (in other words lack of symmetry), the noncompactness of the subgroup $N$ which somewhat takes the place of the maximal compact group $K$ and lack of a rich representation theoretic background.

Purpose of this article is to consider the theorems mentioned above for general functions on Damek-Ricci spaces and to put them in a form which reveals their differences with that on the euclidean spaces.

Two main ingredients of the proof of these theorems in \cite{BP} are analogues of Hausdorff-Young inequality (for radial and $K$-finite functions proved in \cite{EKT,EKu1}) and {\em restriction}  theorem (proved in \cite{MRSS, S-S}); precisely, for $1\le p<2$ and $|\eta|<(2/p-1)\rho$
$$|\wtilde{f}(\lambda+i\eta,k)|dk\le C\|f\|_p.$$
However, it turns out that on symmetric spaces or more generally on $NA$ groups one can have stronger analogues of restriction theorem (see \cite{Ray-Sarkar} and Section 3). In the case of symmetric spaces these results can be interpreted as norm estimates of certain matrix coefficients of  class-$1$ principal series representations of the underlying group $G$ and can be linked to the {Kunze-Stein phenomenon} (see \cite{Cow-Annals}). We need to elaborate on this. Like radial functions the Fourier transform of  a general $L^p$-function also exist on the strip $S_p$ (defined in Section 2) parallel to the real line. But there is an {\em angular} variable ($k$ for symmetric space and  $n$ for $NA$ group).  If $f$ is a radial function in $L^p(X)$ then $\sup_\lambda|\what{f}(\lambda)|\le C\|f\|_p$ in  a smaller strip. But in the absence of radiality one is concerned about  the behavior in $k$ variable of the Fourier transform. We see that the behavior depends on $\Im \lambda$. That is inside the strip on every line parallel to the real axis Fourier transform behaves differently in $k$ variable. We may stress that in particular  it changes as we move from $\alpha+i\eta$ to $\alpha-i\eta$. In the context of symmetric space this can be attributed to Herz's principe de majoration ) and Kunze-Stein phenomenon (\cite{Cow-Annals, Cow-Herz}). 

Coming to the second theorem (which is an application of Hausdorff-Young inequality) we notice that, unlike $\R^n$  on $NA$ the following well known inequality is not true:
$$\sup_{\lambda\in \R}\|\wtilde{f}(\lambda, \cdot)\|_{L^\infty(K)}\le C\|f\|_1.$$
In fact the best result we know is:
$$\sup_{\lambda\in \R}\|\wtilde{f}(\lambda, \cdot)\|_{L^2(K)}\le C\|f\|_1.$$ This indicates that
on these spaces the Hausdorff-Young inequality will involve {\em mixed norms} which will change as  $\lambda$ will vary over a strip on the complex plane. Strips which are symmetric about the real line are the natural domains of the Fourier transforms $\wtilde{f}(\lambda, k)$ for various Lebesgue and Lorentz spaces.
We may stress  here that finer subdivisions of $L^p$ spaces called Lorentz spaces appear naturally in this set up for both of these theorems (see Section 3 for details).

Such results are recently  proved by the authors (\cite{Ray-Sarkar}). In this article we further improve these theorems and as  consequences obtain new analogues of Theorem 1.1 and 1.2. We conclude by noticing that the $L^p$-norm of the spherical mean operator $M_t$ defined in \cite{BP} decays exponentially as $t\rightarrow \infty$ (proved in Section 4 and Section 6). This is a  noneuclidean phenomenon which also vindicates  that in these spaces results will be different from those on the euclidean set up.

\section{preliminaries} Most of the   preliminaries can be found in  \cite{ACB, ADY, Ray-Sarkar}. To make the article self-contained  we shall gather only those  results which are required for this paper. For a detailed account we refer to \cite{Ray-Sarkar}.

Everywhere
in this article for any $p\in [1, \infty)$,  $p'=p/(p-1)$ and $\gamma_p=2/p-1$, $\gamma_\infty=-1$. We note that
$\gamma_p=-\gamma_{p'}$. For a complex number $z$, we will use $\Re z$ and $\Im z$ to
denote respectively the real and imaginary parts of $z$.
For $p\in [1, 2]$ we define $$S_p=\{z\in \C\mid |\Im z|\le \gamma_p\rho\}.$$ By $S_p^\circ$ we denote the interior of the strip.

We will
follow the standard practice of using the letter $C$ for constant,
whose value may change from one line to another. Occasionally the
constant $C$ will be suffixed to show its dependency on important
parameters. The letters $\C$ and $\R$ will denote the set of complex and real numbers respectively.

Let $\mathfrak n=\mathfrak v\oplus \mathfrak z$ be a $H$-type algebra where $\mathfrak v$ and $\mathfrak z$ are vector spaces over $\R$ of dimensions $m$ and $k$ respectively. Indeed $\mathfrak z$ is the centre of $\mathfrak n$ and $\mathfrak v$ is its ortho-complement with respect to the inner product of $\mathfrak n$. Let $N=\exp \mathfrak n$. We shall identify  $\mathfrak v$ and $\mathfrak z$ and $N$ with $\R^m$, $\R^k$ and   $\R^m\times \R^k$ respectively. Elements of $A$ will be identified with $a_t=e^t, t\in \R$. $A$  acts on $N$ by nonisotropic dilation: $\delta_{a_t}(X,Y)=(e^{-t}X, e^{-2t}Y)$. Let $S=NA$ be the semidirect product of $N$ and $A$ under the action above. Then $S$ is a solvable connected and simply connected Lie group with Lie algebra $\mathfrak s=\mathfrak v\oplus\mathfrak z\oplus\R$. It is well known that $S$ is a nonunimodular amenable Lie group. The homogenous dimension of $S$ is $Q=m/2+k$. For convenience we shall also use the symbol $\rho$ for $Q/2$.  An element $x=na=n(X,Y)a\in S$ can be written as $(X, Y, a)$, $X\in \mathfrak v, Y\in \mathfrak z, a\in A$. Precisely $(X, Y, a)$ is identified with $\exp(X+Y)a$. We shall use the notation $A(x)=A(na_t)=t$.

A function $f$ on $S$ is called {\em radial} if for all $x,y\in S$, $f(x)=f(y)$ if $d(x,e)=d(y,e)$, where $d$ is the metric induced by the canonical left invariant riemannian structure of $S$. For a  radial function $f$ we shall also use $f(t)$ to mean $f(a_t)$.

For a suitable function $f$ on $S$ its
radialization $Rf$ is defined as
\begin{equation}Rf(x)=\int_{S\nu}f(y)d\sigma_\nu(y),
\label{radialization}
\end{equation} where $\nu=r(x)$ and $d\sigma_\nu$ is the
surface measure induced by the left invariant Riemannian metric on
the geodesic sphere $S_\nu=\{y\in S\mid d(y, e)=\nu\}$ normalized
by $\int_{S_\nu}d\sigma_\nu(y)=1$. It is clear that $Rf$ is a radial function and if $f$ is
radial then $Rf=f$.

The Poisson kernel ${\mathcal
P}:S\times N\lgra \R$ is given by ${\mathcal P}(na_t,
n_1)=P_{a_t}(n_1^{-1}n)$ where
\begin{equation}P_{a_t}(n)=P_{a_t}(V, Z)=C a_t^{Q}\left(\left(a_t+\frac{|V|^2}4\right)^2+|Z|^2\right)^{-Q},\,\, n=(V,
Z)\in N.
\label{poisson}\end{equation}  The  value of $C$ is adjusted so that $\int_NP_a(n)dn=1$ and $P_1(n)\le 1$ (see \cite[(2.6)]{ACB}).
We also need the following:
\begin{enumerate}
\item  $P_a(n)=P_a(n^{-1})$.
\item $P_{a_t}(n)=P_1(a_{-t}na_t)e^{-2\rho t}$.
\item $\PP(x,n)=\PP(n_1a_t,n)=P_{a_t}(n^{-1}n_1)=P_{a_t}(n_1^{-1}n)$.
\item $\PP_\lambda(x, n)=\PP(x,n)^{1/2-i\lambda/Q}=\PP(x,n)^{-(i\lambda-\rho)/Q}$
\item $R({\mathcal P}_\lambda (\cdot, n))(x)=\phi_\lambda(x){\mathcal P}_\lambda (e, n)$, $R(e^{(i\lambda-\rho)A(\cdot)})(x)=\phi_\lambda(x)$.
\end{enumerate}

The action of class-$1$ principal series  representation $\pi_\lambda, \lambda\in \C$ realized on functions on $N$ is given by: $$(\pi_{-\lambda}(n_1a_t)\phi)(n)=\phi(a_{-t}n_1^{-1}na_t)e^{t(i\lambda-\rho)}.$$
From this it is easy to verify  that $(\pi_{-\lambda}(x)P_1^{1/2-i\lambda/Q})(n)=\PP_\lambda(x,n)$.

The elementary spherical function $\phi_\lambda(x)$ is given by
$$\phi_\lambda(x)=\langle \pi_\lambda(x)P_1^{1/2-i\lambda/Q}, P_1^{1/2-i\lambda/Q}\rangle_{L^2(N)} =
\int_N\PP_\lambda(x,n)\PP_{-\lambda}(e,n)dn.$$ It follows that $\phi_\lambda$ is a radial eigenfunction of the Laplace-Beltrami operator $\mathcal L$ of $S$ with eigenvalue $-(\lambda^2+\rho^2)$ satisfying $\phi_\lambda(x)=\phi_{-\lambda}(x), \phi_\lambda(x)=\phi_\lambda(x^{-1})$ and $\phi_\lambda(e)=1$.
As $\PP_{-i\rho}(x,n)\equiv 1$ for all $x\in S$ and $n\in N$  and $\PP_{i\rho}(x,n)=\PP(x,n)$, $$\phi_{-i\rho}(x)=\int_N\PP_{i\rho}(e,n)=\int_NP_1(n)dn=1.$$

For $\alpha=\frac{m+k-1}2$ and $\beta=\frac{k-1}2$,   $\phi_\lambda$ is identical with the Jacobi function $\phi^{(\alpha, \beta)}_\lambda$ with the {\em ideal situation} of $\alpha>\beta>-\frac 12$ (see\cite{ADY}).
Thus spherical Fourier transform is related to the Jacobi transform.

We define the spherical Fourier transform $\what{f}$ of a suitable radial
function $f$ as
\begin{equation*}\what{f}(\lambda)=\int_Sf(x)\phi_\lambda(x)dx,\end{equation*} whenever the integral converges.

The left invariant Haar measure on $S$ decomposes as $$\int_S f(x)dx=\int_{N\times A}f(na_t)e^{-2\rho t}dtdn,$$ where $dn(X,Y)=dX\,dY$ and $dX, dY, dt$ are Lebesgue measures on $\mathfrak v$, $\mathfrak z$ and $\R$ respectively.

Jacobians of the following transformations will be required for our computations.
\begin{enumerate}
\item[(a)] $\int_Nf(a_tna_{-t})=\int_N f(n)e^{-2\rho t}dn$.
\item[(b)]  $\int_SR_yf(x)dx=\int_Sf(xy)dx=\int_Sf(x)dxe^{2\rho A(y)}$, i.e. the modular function $\Delta(y)=e^{-2\rho A(y)}$. Here $R_y$ denotes the right-translation operator.
\item[(c)] $\int_Sf(x^{-1})dx=\int_Sf(x)e^{2\rho A(x)}dx$ and $\int_Sf(x^{-1})e^{2\rho A(x)}dx=\int_Sf(x)dx$.
\end{enumerate}

For two measurable functions $f$ and $g$ on $S$ we define their convolution as (see \cite[p. 51]{Folland}):
$$f\ast g(x)=\int_Sf(y)g(y^{-1}x)dy=\int_Sf(y^{-1})g(yx)\Delta(y^{-1})dy=\int_Sf(xy^{-1})g(y)\Delta(y^{-1})dy.$$
For a measurable  function $f$ on $S$ we define its  Fourier transform (which is an analogue of the Helgason Fourier transform on the symmetric space) by $$\wtilde{f}(\lambda, n)=\int_Sf(x)\PP_\lambda(x, n)dx,$$ whenever the integral converges.
If $f$ is radial then using (5) above we see that
$\wtilde{f}(\lambda, n)=\what{f}(\lambda)\PP_\lambda(e, n)$.


The Poisson transform of a function $F$ on $N$ is defined as (see
\cite{ACB})
\begin{equation*}{\mathfrak P}_{\lambda}F(x)=\int_N F(n){\mathcal P}_\lambda (x, n)dn.\end{equation*}
Any norm estimate involving the Fourier transform of a function is equivalent to a dual statement involving the Poisson transform. Precisely, for a function $f$  on $S$, a function $F$ on $N$ and for $\lambda\in \C$,
$$\|\wtilde{f}(\lambda, \cdot)\|_{L^q(N)}\le C\|f\|_p\iff \|{\mathfrak P}_{\lambda}F\|_{L^{p'}(S)}\le C
\|F\|_{L^{q'}(N)}.$$

\vspace{2ex}
We shall denote the $(p,q)$-Lorentz spaces  by $L^{p,q}(S)$ and the corresponding norm by $\|\cdot\|^\ast_{p,q}$. We recall that $L^{p,\infty}(S)$ is called weak $L^p$-space. For definitions and other details  on Lorentz spaces we refer to \cite{Graf, S-W, Ray-Sarkar}.

\section{Existence and some properties of the Fourier transform} The following two theorems are proved by the authors
in \cite{Ray-Sarkar}.
\begin{theorem}
Let $f$ be a measurable function in the Lorentz space $L^{p,
q}(S)$.
\begin{enumerate}
\item[(i)] If $1\le p<2$ and $q=1$ then there exists a subset
$N^p$ of $N$ of full Haar measure, depending only on $f$, such
that $\wtilde{f}(\lambda, n)$ exists for all $n\in N^p$ and
$\lambda\in S_p$.

\item[(ii)] If $1<p<2$ and $1<q \le \infty$ then there exists a
subset $N^p$ of $N$ of full Haar measure, depending only on $f$,
such that $\wtilde{f}(\lambda, n)$ exists for all $n\in N^p$ and
$\lambda\in S^\circ_p$.

\item[(iii)] If  $p, q$ are as in {\rm (ii)} then there exists a
subset $N'_p$ of $N$ of full Haar measure, depending only on $f$,
such that $\wtilde{f}(\lambda, n)$ exists for all $n\in N'_p$ and
almost every $\lambda\in \partial S_p$.
\end{enumerate}
\label{HFT-existence}
\end{theorem}

\begin{theorem}[Riemann-Lebesgue Lemma]
Let $1\le p<2$. If $f\in L^{p,1}(S)$   then for almost every fixed
$n\in N$ the map $\lambda\mapsto \wtilde{f}(\lambda, n)$ is
continuous on $S_p$ and analytic on $S_p^\circ$. Furthermore
\begin{equation*}\lim_{|\xi|\rightarrow\infty}\wtilde{f}(\xi+i\eta, n)=0\end{equation*} uniformly in
$\eta\in [-\gamma_p\rho, \gamma_p\rho]$.

For functions in $L^{p,q}(S)$, $q>1$ the assertions above remain
valid for $\lambda\in S_p^\circ$ and for $\eta\in
[-(\gamma_{p}\rho-\delta), (\gamma_{p_1}\rho-\delta)]$ for any
$0<\delta<\gamma_p$. \label{R-L}
\end{theorem}

Here are  improved versions of some relevant theorems proved in \cite{Ray-Sarkar}:

\begin{theorem}[Restriction on line]

For $f\in L^1(S)$, $q\in [1, \infty]$ and $\alpha\in \R$,
\begin{equation*}\left(\int_N|\wtilde{f}(\alpha+i\gamma_q\rho, n)|^q dn\right)^{1/q}\le \|f\|_1.\end{equation*}

For $f\in L^{p, \infty}(S), 1<p<2$,  $p<q<p'$ and $\alpha\in \R$,
\begin{equation*}\left(\int_N|\wtilde{f}(\alpha+i\gamma_{q}\rho, n)|^q dn\right)^{1/q}\le C_{p,q}\|f\|^\ast_{p, \infty}.\end{equation*}
\label{restriction-thm-1}
\end{theorem}
\begin{proof}
We will prove only the second part. We take $p_1, p_2\ge 1$ such that $p_1<p<p_2<q<p'$. Using the result in \cite[Theorem 4.2]{Ray-Sarkar} we have
\begin{eqnarray*}
\|\wtilde{f}(\alpha+i\gamma_q\rho, \cdot)\|_{L^q(N)}\le C_{p_1,q}\|f\|_{p_1}
\end{eqnarray*}
which is equivalent to the following by duality:
\begin{eqnarray*}
\|\mathfrak P_{\alpha+i\gamma_q\rho}\xi\|_{p_1'}\le C_{p_1,q}\|\xi\|_{L^{q'}(N)}.
\end{eqnarray*}
Through similar arguments we also get
\begin{eqnarray*}
\|\mathfrak P_{\alpha+i\gamma_q\rho}\xi\|_{p_2'}\le C_{p_2,q}\|\xi\|_{L^{q'}(N)}.
\end{eqnarray*}
We interpolate between the two results above (\cite[p. 64, 1.4.2]{Graf}) to get
\begin{eqnarray*}
\|\mathfrak P_{\alpha+i\gamma_q\rho}\xi\|^\ast_{p',1}\le C_{p_1, p_2,p,q}\|\xi\|_{L^{q'}(N)}.
\end{eqnarray*} as $p_2'<p'<p_1'$. The last result is equivalent to (by duality)
\begin{eqnarray*}
\|\wtilde{f}(\alpha+i\gamma_q\rho, \cdot)\|_{L^q(N)}\le C_{p_1, p_2,p,q}\|f\|^\ast_{p, \infty}.
\end{eqnarray*}
\end{proof}

It is clear that for $p,q,\alpha$ as in the second part of the theorem above and $1\le r<\infty$, if $f\in L^{p,r}(S)$ then
\begin{equation*}\left(\int_N|\wtilde{f}(\alpha+i\gamma_{q}\rho, n)|^q dn\right)^{1/q}\le C_{p,q}\|f\|^\ast_{p, r}.\end{equation*}

To have a norm estimate of $\wtilde{f}(\lambda, \cdot)$ which is uniform over the strip  $S_q$, we  consider a weighted measure space $(N, P_1(n)dn)$.

\begin{corollary}[Restriction on strip] Let $$L^{q}(N, P_1)=\{f \text{ measurable on } N
\mid\int_N|f(n)|^qP_1(n) dn<\infty\}.$$

{\em (a)} Let $1\le p<q\le 2$ and  $1\le r\le q$. If $f\in L^{p, \infty}(S)$ then
\begin{equation*}\|\wtilde{f}(\lambda, \cdot)\|_{L^{r}(N, P_1)}\le C_{p,q}\|f\|^\ast_{p, \infty}\end{equation*} for
any $\lambda$ in the strip ${\mathcal S}_{q}=\{z\in \C\mid |\Im
z|\le \gamma_{q}\rho\}.$

{\em (b)} Let $1\le p<q<2$ and $f\in L^{p,\infty}(S)$. Then $$\|\wtilde{f}(\lambda, \cdot)\|_{L^{q,1}(N, P_1)}\le C_{\lambda,p,q}\|f\|^\ast_{p,\infty}\text{ for all }\lambda\in S_q^\circ.$$

{\em (c)} For $p<q<q_1\le 2$, $\lambda\in \R$
\begin{equation*}\|\wtilde{f}(\lambda+i\gamma_{q_1}, \cdot)\|^\ast_{L^{q,1}(N, P_1)}\le C_{p,q,q_1}\|f\|^\ast_{p, \infty}\end{equation*}
\label{cor-restriction}
\end{corollary}

\begin{proof}  Applying the arguments of \cite[Corollary 4.4]{Ray-Sarkar} on Theorem \ref{restriction-thm-1}  we get (a).

For (b) we take a $\lambda\in S_q^\circ$. Then  $\lambda=\alpha+i\gamma_{q_1}\rho$ for some $q<q_1\le 2$. We choose a $q_2$ such that $1\le q_2 <q<q_1\le 2$. By Theorem \ref{restriction-thm-1} and as $P_1(n)\le 1$
$$\|\wtilde{f}(\alpha+i\gamma_{q_1}\rho, \cdot)\|_{L^{q_1}(N, P_1)}\le C_{p,q_1}\|f\|^\ast_{p,\infty}.$$ As   $(N, P_1(n))$ is a finite measure space, this implies
$$\|\wtilde{f}(\alpha+i\gamma_{q_1}\rho, \cdot)\|_{L^{q_2}(N, P_1)}\le C_{p,q_1}\|f\|^\ast_{p,\infty}.$$ An interpolation (\cite[p.64]{Graf}) between these two results gives (b).

In (c) we make the constant independent of $\lambda$ by fixing $q_1$.
\end{proof}

We consider the product measure space $(Y, dy)=(N, dn)\times (\R, |c(\lambda)|^{-2}d\lambda)$. For a measurable function $F(\lambda, n)$ on this measure space we denote mixed norm of $F$ by
$$\|F\|^\ast_{(q; p',r)}=\left\|\left(\int_N|F(\cdot, n)|^qdn\right)^{1/q}\right\|^\ast_{p',r}.$$
\begin{theorem}[Hausdorff-Young Theorem] Let  $1\le p\le 2$. Then

{\em (a)} for  $p\le q\le p'$,
\begin{equation*}\left(\int_\R\left(\int_N|\wtilde{f}(\lambda+i\gamma_q\rho, n)|^qdn\right)^{\frac{p'}q}
|c(\lambda)|^{-2}d\lambda\right)^{\frac 1{p'}}\le
C_{p,q}\|f\|_p.\end{equation*}

{\em (b)} for $p< q < p'$, $p'\le r\le \infty$ and $1\le s\le \infty$
$$\|\wtilde{f}(\cdot+i\gamma_q,\cdot)\|^\ast_{(q;r,s)}\le C_{p,q,r}\|f\|^\ast_{p,s}.$$

\label{H-Y}
\end{theorem}
\begin{proof} Part (a) is proved in \cite[Theorem 4.6]{Ray-Sarkar}.

For (b)
we consider the operator $T$ between the measure spaces $(S, dx)$ and $(\R, |(c(\lambda)|^{-2}d\lambda)$ defined by:
$$Tf(\lambda)=\|\wtilde{f}(\lambda+i\gamma_q\rho, \cdot)\|_q.$$ We choose $p_1$ and $p_2$ such that $1\le p_1<p<p_2<q$. Then by  \cite[Corollary 4.7 ]{Ray-Sarkar}
$\|Tf\|_{p_1'}\le C_{p,q}\|f\|_{p_1}$ and $\|Tf\|_{p_2'}\le C_{p,q}\|f\|_{p_2}$. Interpolating them (\cite[p.197]{S-W}))
we get (b).
\end{proof}

We note that part (b) of the theorem above generalizes an euclidean result (see \cite[p. 200]{S-W}) where teh inequality is proved for $r=s=p'$.

\section{The spherical mean operators}
Let $\sigma_t$ be the normalized surface measure of the geodesic sphere of radius $t$. For a suitable function $f$ on $S$ we define the spherical mean operator $M_tf=f\ast \sigma_t$.
Using the radialization operator $R$ (see Section 2) then $$M_tf(x)=R(\, {}^x\!f)(t)$$ where ${}^x\!f$ is the right-translation of $f$ by $x$.

We need to make the radialization operator more precise.
We note that if $d(x, e)=t$ where $x=na_r\in S$ and $n=(X, Y)$ then
$$(\cosh t)^2=\left[\cosh r+e^r|X|^2\right]^2+e^{2r}|Y|^2.$$
We define the surface for any $t\ge |s|$,
$$T_{t,s}=\{(X, Y)\in \R^m\times \R^k\mid (\cosh t)^2=\left[\cosh r+e^r|X|^2\right]^2+e^{2r}|Y|^2\}.$$ Then $T_{t,s}$ is the set of points $P=P(X, Y)\in \R^m\times\R^k=N$ such that $d(Pa_s, e)=t$. Let $dw_{t,s}$ be the induced measure on $T_{t,s}$ such that for a suitable function $\Phi$ on $\R^m\times \R^k$,
$$\int_{\R^m\times\R^k}\Phi(X,Y)dXdY=\int_{t\ge|s|}\left[\int_{T_{t,s}}\Phi(P)dw_{t,s}(P)\right]dt.$$

Then the radialization opertaor $R$ can be defined by the following:
$$R(\Phi)(t)=\int_{|s|<t}\left[\int_{T_{t,s}}\Phi(X,Y, a_s)dw_{t,s}(X,Y)\right]e^{-2\rho s}ds.$$ Using this expression of radialization we shall prove the $(p,p)$ property of $M_t$.

\begin{proposition}
For $1\le p\le\infty$ $$\|M_tf\|_p\le \phi_{i\gamma_p\rho}(a_t)\|f\|_p.$$
\end{proposition}
\begin{proof} For convenience let us denote the variable point $(X, Y, a_s)$ in the integration defining radialization $R$ simply by $P$. We recall that $M_t(f)(x)=R(\,^x\!f)(t)$.
\begin{eqnarray*}
&&\left(\int_S|M_tf(x)|^pdx\right)^{1/p}\\
&=&\left(\int_S\left |\int_{|s|<t}\int_{T_{t,s}}\,^x\!f(P)dw_{t,s}\,e^{-2\rho s}ds\right |^pdx\right)^{1/p}\\
&\le&\int_{|s|<t}\int_{T_{t,s}}\left(\int_S \,|^x\!f(P)|^pdx\right)^{1/p}dw_{t,s}\,e^{-2\rho s}ds\\
&=&\int_{|s|<t}\int_{T_{t,s}}\left(\int_S \,|f(x)|^p e^{2\rho s}dx\right)^{1/p}dw_{t,s}\,e^{-2\rho s}ds\\
&=&\|f\|_p \int_{|s|<t}\int_{T_{t,s}}e^{2\rho s/p}\, dw_{t,s}\,e^{-2\rho s}ds\\
&=& \|f\|_p\, \phi_{i\gamma_p\rho}(a_t).
\end{eqnarray*}
In the last step we have used that $e^{2\rho s/p}=e^{2\rho A(P)/p}$ and as $1/p=(1/2-i(i\gamma_p\rho)/2\rho)$, $R(e^{2\rho A(\cdot)/p})(t)=\phi_{i\gamma_p\rho}(a_t)$ (see Section 2).
\end{proof}
Since for $t>0$  $\phi_{i\gamma_p\rho}(a_t)\asymp e^{-(2\rho/p')t}$ for $1\le p\le 2$ and $\phi_{i\gamma_{p'}\rho}=\phi_{i\gamma_p\rho}$, we have from above $\|M_tf\|_{op}\le e^{-(2\rho/p')t}$ or
$\le e^{-(2\rho/p)t}$ depending on $p\le 2$ or $p>2$. Here $\|M_tf\|_{op}$ is the operator norm of $M_t$ from $L^p(S)$ to $L^p(S)$.
The proof of the  proposition above for symmetric space is given in Section 6.

Using interpolation (\cite[p.197]{S-W}) we have $$\|M_tf\|^\ast_{p,s}\le C_{p,s}\|f\|^\ast_{p,s}$$ for $p\in (1, \infty)$, $s\in [1,\infty]$.

\begin{proposition} For $f\in L^1(S)$ $M_tf$ converges to $f$ in $L^1$ as $t\rightarrow 0$.
Also for all $f\in L^{p,q}(S), 1<p<\infty$, $1\le q\le \infty$ $M_tf$ converges to $f$ in $L^p$ as $t\rightarrow 0$.
\end{proposition}
A standard argument involving dominated convergence theorem and approximation by functions in $C_c^\infty(S)$ proves the result for $L^p$-spaces.  If $f\in L^{p,q}(S)$ with $p, q$ as above, then there exists $f_1\in L^1(S), f_2\in L^r(S)$ with $r\in (p,2]$ such that $f=f_1+f_2$ (see \cite{Ray-Sarkar}). Use of this decomposition gives the result for Lorentz spaces.

\begin{proposition}
For  $f \in L^p(S), 1\le p<2$, $(M_tf)\wtilde{\,\,}(\lambda, n)=\wtilde{f}(\lambda, n)\phi_\lambda(t)$.
\end{proposition}
\begin{proof}
We note that
\begin{eqnarray}
\int_Sf(xy)\PP_\lambda(x, n_1)dx&=& \int_Sf(xy)(\pi_{-\lambda}(x)P_1^{1/2-i\lambda/Q})(n_1)dx\nonumber\\
&=&\int_Sf(z)(\pi_{-\lambda}(zy^{-1})P_1^{1/2-i\lambda/Q})(n_1)e^{2\rho A(y)}dz\nonumber\\
&=&\int_Sf(z)(\pi_{-\lambda}(z)(\pi_{-\lambda}(y^{-1})P_1^{1/2-i\lambda/Q}))(n_1)dz\, e^{2\rho A(y)}\nonumber\\
&=&(\pi_{-\lambda}(f)\PP_\lambda(y^{-1}, \cdot))(n_1)e^{2\rho A(y)}.
\label{HFT-right-translation}
\end{eqnarray}

In the  computations below we shall use the following substitutions in different steps:
$y=na_s$ and $x=n_2a_r$, $n_3=a_{-s}n^{-1}a_s$, $n'=a_{-r}n_2^{-1}n_1a_r$.  We shall also use the fact that
  $R\PP_\lambda(\cdot, n)(t)=\PP_\lambda(e, n)\phi_\lambda(t)$ (see section 2).

\begin{eqnarray*}
(M_tf)\wtilde{\,\,}(\lambda, n_1))&=&\int_SM_tf(x)\mathcal P_\lambda(x, n_1)dx\\
&=&\int_{|s|\le t}\int_{T_{t,s}}\int_S f(xy)\PP_\lambda(x, n_1)dx dw_{t,s}(n)e^{-2\rho s}ds\\
&=&\int_{|s|\le t}\int_{T_{t,s}}\left[(\pi_{-\lambda}(f)\PP_\lambda(y^{-1}, \cdot))(n_1)e^{2\rho s}\right ]dw_{t,s}(n)e^{-2\rho s}ds\\
&=&\int_{|s|\le t}\int_{T_{t,s}}\left [\int_S f(x)(\pi_{-\lambda}(x)\PP_\lambda(y^{-1}, \cdot))(n_1)e^{2\rho s}\, dx\right] dw_{t,s}(n)e^{-2\rho s}ds\\
&=&\int_{|s|\le t}\int_{T_{t,s}}\left [\int_{N\times\R} f(n_2a_r)\PP_\lambda(y^{-1}a_{-r}n_2^{-1}n_1a_r)e^{r(i\lambda-\rho)}e^{2\rho s}dx\right] dw_{t,s}(n)e^{-2\rho s}ds\\
&=&\int_Sf(x)\left[\int_{|s|\le t}\int_{T_{t,s}}\PP_\lambda(y^{-1}n')e^{2\rho s}dw_{t,s}(n)e^{-2\rho s}ds\right]e^{r(i\lambda-\rho)}\, e^{-2\rho r} dn_2\,dr\\
&=&\int_{N\times\R}f(n_2a_r)\left[\int_{|s|\le t}\int_{T_{t,s}}\PP_\lambda(n_3a_{-s},n')e^{4\rho s}dw_{t,s}(n_3)e^{-2\rho s}ds\right]
e^{r(i\lambda-\rho)}\, e^{-2\rho r} dn_2\,dr\\
&=&\int_{N\times\R}f(n_2a_r)\left[R(\PP_\lambda(\cdot, n'))(t)\right]
e^{r(i\lambda-\rho)}\, e^{-2\rho r} dn_2\,dr\\
&=&\int_{N\times\R}f(n_2a_r)\PP_\lambda(e, n')\phi_\lambda(t)e^{r(i\lambda-\rho)}\, e^{-2\rho r} dn_2\,dr\\
&=&\int_{N\times\R}f(n_2a_r)(\pi_{-\lambda}(x)P_1^{1/2-i\lambda/Q})(n_1)dx \phi_\lambda(t)\text{ (by \ref{HFT-right-translation})}\\
&=&\wtilde{f}(\lambda, n_1)\phi_\lambda(t).
\end{eqnarray*}
\end{proof}

Following Bray we define {\em spherical modulus of continuity} for any $1\le p,q\le \infty$ as
$$\Omega_{p,q}[f](r)=\sup_{0<t\le r}\|M_tf-f\|^\ast_{p,q}.$$ We note that $\Omega_{p,p}$ is the same as $\Omega_p$.
We shall  quote two lemmas from \cite{BP}. Let $j_\alpha$ be the usual Bessel function of the first kind normalized by $j_\alpha(0)=1$.
\begin{lemma} For $\alpha>-1/2$
$$C_{1, \alpha}\min\left\{1,\left(\frac \lambda r\right)^2\right\}\le \int_0^1\left(1-j_\alpha\left(\frac{\lambda z}r\right)\right)dz\le \sup_{0\le z\le 1}\left(1-j_\alpha\left(\frac{\lambda z}r\right)\right)\le  C_{2, \alpha}\min\left\{1,\left(\frac \lambda r\right)^2\right\}$$ for two positive constants  $C_{1, \alpha}$ and  $C_{2, \alpha}$.
\label{BP-lemma-1}
\end{lemma}

\begin{lemma} Let $\alpha\ge \beta\ge -1/2$, $t_0>0$ and $|\eta|\le \rho$. Then for all $0\le t\le t_0$,
$$|1-\phi_{\mu+i\eta}^{(\alpha, \beta)}(a_t)|\ge C|1-j_\alpha(\mu t)|$$ for some positive constant $C=C_{t_0, \alpha, \beta}$.  Consequently
$$\int_0^1\left|1-\phi_{\mu+i\eta}^{(\alpha, \beta)}\left(\frac zr\right)\right|dz\ge C\min\left\{1, \left(\frac\mu r\right)^2\right\}.$$
\label{BP-lemma-2}
\end{lemma}

\section{Growth of Fourier transform and Moduli of continuity}
We offer the following modification of two main Theorems in \cite{BP} mentioned in the introduction.
\begin{theorem}\label{bray-thm-restriction}
 Let $r\ge r_0>0$ be fixed.

{\em(a)} Let $q\in [1, \infty]$. Then for $f\in L^1(S)$ and $\lambda\in \R$
$$\sup_\lambda\left[\min\left\{1, \left(\frac{\lambda}r\right)^2\right\}\left(\int_N|\wtilde{f}(\lambda+i\gamma_q\rho,n)|^qdn\right)^{1/q}\right]\le C_q\Omega_1[f]\left(\frac 1r\right).$$

{\em(b)} Let $1<p<2, p<q<p'$. Then for $f\in L^{p, \infty}(S)$ and $\lambda\in \R$
$$\sup_\lambda\left[\min\left\{1, \left(\frac{\lambda}r\right)^2\right\}\left(\int_N|\wtilde{f}(\lambda+i\gamma_q\rho,n)|^qdn\right)^{1/q}\right]\le C_{p,q}\Omega_{p, \infty}[f]\left(\frac 1r\right).$$

{\em(c)} Let $1\le p<q\le 2$ and $f\in L^{p, \infty}(S)$. Then for $|\eta|<\gamma_p\rho$  and $\lambda\in \R$
$$\sup_\lambda\left[\min\left\{1, \left(\frac{\lambda}r\right)^2\right\}\left(\int_N|\wtilde{f}(\lambda+i\eta,n)|^qP_1(n)dn\right)^{1/q}\right]\le C_{p,q}\Omega_{p, \infty}[f]\left(\frac 1r\right).$$
\end{theorem}

The proofs follow from Theorem \ref{restriction-thm-1}, Corollary \ref{cor-restriction} (a) and  the arguments  of the corresponding result in \cite{BP}. We omit it for brevity. One can also prove similar results using part (b) and (c) of Corollary \ref{cor-restriction}.

We need the following lemma.
\begin{lemma} Let $1<p\le 2$ and $1\le q \le \infty$
Let $g$ be a nonnegative bounded continuous function on $[0,1]\times S$ and $f$ be a nonnegative function in $L^{p,q}(S)$. Then,
$$\left\|\left(\int_0^1g(t,\cdot))dt\right)f\right\|^\ast_{p,q}\le \sup_{[0,1]}\|g(t, \cdot)f\|^\ast_{p,q}.$$
\label{lem-Lorentz}
\end{lemma}

The proof uses a standard  argument involving duality and Fubini's theorem.

\begin{theorem} \label{bray-thm-HY}
{\em (a)} Let $1\le p\le 2$
 and $p\le q\le p'$. Then for $f\in L^p(S)$
   $$\left(\int_\R\min\left\{1, \left(\frac{\lambda}r\right)^{2p'}\right\}\left(\int_N|\wtilde{f}(\lambda+i\gamma_q\rho,n)|^{q}dn\right)^{p'/q}|c(\lambda)|^{-2}d\lambda
   \right)\le C_{p,q}\Omega_{p}[f]\left(\frac 1r\right).$$

   {\em (b)} Let $1<p\le 2$, $p<q<p'$ and $1\le s \le \infty$. Then for $f\in L^{p,s}(S)$ and $\alpha\in [p',\infty]$,
   $$\left\|\min\left\{1, \frac{\cdot}{r}\right\}^2\left(\int_N|\wtilde{f}(\mathbf\cdot+i\gamma_q\rho,n)|^{q}dn\right)^{1/q}\right\|^\ast_{\alpha, s}
   \le C_{p,q,r,s}\Omega_{p,s}[f]\left(\frac 1r\right).$$

\end{theorem}

\begin{proof}

We shall prove only (b).
(b) We apply  Theorem \ref{H-Y} (b) on the function $M_tf-f\in L^{p,s}(S)$ and subsequently use Lemma \ref{BP-lemma-2}, Lemma \ref{BP-lemma-1} and Lemma \ref{lem-Lorentz} to get the result through the following steps.

\begin{eqnarray*}
&&\left\|\,|\phi_{\cdot+i\gamma_q\rho}(a_t)-1|\left(\int_N|\wtilde{f}
(\mathbf{\cdot}+i\gamma_q\rho,n)|^{q}dn\right)^{1/q}\right\|^\ast_{\alpha, s} \le C_{p,q,r,s}\|M_tf-f\|^\ast_{
p,s}.\end{eqnarray*}
Form this we get
\begin{eqnarray*}
&& \sup_{0\le t\le 1/r}\left\|\,|\phi_{\mathbf{\cdot}+i\gamma_q\rho}(a_t)-1|\left(\int_N|\wtilde{f}(\mathbf\cdot+i\gamma_q\rho,n)|^{q}dn
\right)^{1/q}\right\|^\ast_{\alpha, s}    \le C_{p,q,r,s}\Omega_{p,s}[f]\left(\frac 1r\right).
\end{eqnarray*}
This implies
\begin{eqnarray*}
&& \sup_{0\le z\le 1}\left\|\,|1-j_\alpha(\cdot \frac zr)|\left(\int_N|\wtilde{f}(\mathbf{\cdot}+i\gamma_q\rho,n)|^{q}dn\right)^{1/q}\right\|^\ast_{\alpha, s}    \le C_{p,q,r,s}\Omega_{p,s}[f]\left(\frac 1r\right).
\end{eqnarray*} From this using Lemma \ref{lem-Lorentz} we get
\begin{eqnarray*}
&& \left\|\left(\int_0^1|1-j_\alpha(\cdot \frac zr)|dz\right)\left(\int_N|\wtilde{f}(\mathbf{\cdot}+i\gamma_q\rho,n)|^{q}dn\right)^{1/q}\right\|^\ast_{\alpha, s}    \le C_{p,q,r,s}\Omega_{p,s}[f]\left(\frac 1r\right)
\end{eqnarray*} which implies
\begin{eqnarray*}
&& \left\|\min\left\{1,\left(\frac \cdot r\right)^2\right\}\left(\int_N|\wtilde{f}(\mathbf{\cdot}+i\gamma_q\rho,n)|^{q}dn\right)^{1/q}\right\|^\ast_{\alpha, s}    \le C_{p,q,r,s}\Omega_{p,s}[f]\left(\frac 1r\right).
\end{eqnarray*}
Through similar arguments we can prove (a) applying Theorem \ref{H-Y} (a) on the function $M_tf-f$
\end{proof}

\section{Appendix}
In this section we consider the noncompact riemannian symmetric spaces $X$ of rank one. Most of the notations are standard and can be found in \cite{BP}. It is not difficult to see that all the theorems proved for Damek-Ricci spaces will have analogue for symmetric spaces where $N$ will be replaced by $K$ and the Fourier transform defined in Section 2 will be substituted by the usual Helgason Fourier transform. As $K$ is compact and hence a finite measure space, some of the statements will look simpler here, e.g. $P_1(n)$ will be substituted by $1$. We shall omit these results here except for  one additional result for symmetric spaces (Theorem \ref{growth-by-L-R}). We begin with a proof of the norm estimate of $M_t$.

For a function $f$ on $G/K$, let $M_t fx)=\int_Kf(xka_t)dk$. Then  $M_tf$ is a right $K$-invariant function and hence a function on $G/K$. We will see below that $M_t$ is a bounded operator from $L^p(G/K)$ to $L^p(G/K)$ for every $p\ge 1$ and $\|M_t\|_{op}\le e^{-(2\rho/p')t}$ or $e^{-(2\rho/p)t}$ depending on $p\le 2$ or $p>2$. Here $\|\cdot\|_{op}$ denotes the operator norm.
\begin{proposition}
$M_t$ is strong type $(p,p)$ and $\|M_tf\|_p\le \|f\|_p\phi_{i\gamma_p\rho}(a_t).$
\end{proposition}
\begin{proof}
If $x=nak$ then $M_tf(x)=M_tf(na)$. Thus $\int_G|M_tf(x)|^pdx=\int_{N\times \R}|M_tf(n_2a_s)|^p e^{-2\rho s}dn_2ds$.
Therefore \begin{equation}\label{comeback-1-G/K}
\|M_tf\|_p\le \int_K\left(\int_X|f(xka_t)|^pdx\right)^{1/p}dk= \int_K\left(\int_{N\times\R}|f(n_2a_ska_t)|^pe^{-2\rho s}dn_2ds\right)^{1/p}dk.
\end{equation}
For the inside integral we put $ka_t=n_1a_rk_1$. (Then $H(a_t^{-1}k^{-1})=-r$.)
\begin{eqnarray*}
&&\int_{N\times \R}|f(n_2a_ska_t)|^pe^{-2\rho s}dn_2ds\\
&=&\int_{N\times \R}|f(n_2a_sn_1a_r)|^pe^{-2\rho s}dn_2ds\\
&=&\int_{N\times \R}|f(n_2n_3a_{s+r})|^pe^{-2\rho s}dn_2ds \text{ where } n_3=a_sn_1a_{-s}\\
&=&\int_{N\times \R}|f(na_{s})|^pe^{-2\rho s}e^{2\rho r}dnds\\
&=&\|f\|_p^p\, e^{2\rho r}=\|f\|_p^p\, e^{-2\rho H(a_t^{-1}k^{-1})}.
\end{eqnarray*}
We put this back in (\ref{comeback-1-G/K}) to get $\|M_tf\|_p\le \|f\|_p\,\phi_{i\gamma_p\rho}(a_t).$
\end{proof}
 The proposition above  has the following interesting corollary.

\begin{corollary} For any $p\in [1,2)$ if $\lambda\in S_p^\circ$ then $|\phi_\lambda(a_t)|\le \phi_{i\gamma_p\rho}(a_t)$.
\end{corollary}
Note that there is no constant in the inequality.
\begin{proof}
We note that $M_t\phi_\lambda(x)=\int_K\phi_\lambda(xka_t)dk=\phi_\lambda(x)\phi_\lambda(a_t)$.
For $p\in [1,2)$, we take $\lambda\in \mathcal S_p^\circ$.  Then $\phi_\lambda\in L^{p'}(S)$. Using the Proposition above we see that,
$\|M_t\phi_\lambda\|_{p'}\le \phi_{i\gamma_p}(a_t)\|\phi_\lambda\|_{p'}$ and hence
$|\phi_\lambda(a_t)|\|\phi_\lambda\|_{p'}\le \phi_{i\gamma_p}(a_t)\|\phi_\lambda\|_{p'}$. Thus
$|\phi_\lambda(a_t)|\le \phi_{i\gamma_p}(a_t)$ for all $\lambda\in S_p^\circ$.
\end{proof}

On rank one symmetric space $X=G/K$ by  \cite[Lemma 1]{Lo-Ry} we have the following additional result.
\begin{theorem}\label{growth-by-L-R} If $1<p<2$, then
 $$\sup_\lambda\left[\min\left\{1, \left(\frac{\lambda}r\right)^2\right\}\left(\int_K|\wtilde{f}(\lambda+i\gamma_p\rho,k)|^pdk\right)^{1/p}\right]\le C_p\Omega_{p,1}[f]\left(\frac 1r\right)$$ and

  $$\sup_\lambda\left[\min\left\{1, \left(\frac{\lambda}r\right)^2\right\}\left(\int_K|\wtilde{f}(\lambda-i\gamma_p\rho,k)|^{p'}dk\right)^{1/p'}\right]\le C_p\Omega_{p,1}[f]\left(\frac 1r\right).$$

\end{theorem}


\bibliographystyle{amsplain}

\end{document}